\documentclass[
final
 , nomarks
]{dmtcs-episciences}

\usepackage[utf8]{inputenc}
\usepackage{subfigure}

\usepackage{amsmath,amssymb,amsthm}
\usepackage{tikz}

\newtheorem{theorem}{Theorem}
\newtheorem{lemma}{Lemma}
\newtheorem{claim}{Claim}
\newtheorem{case}{Case}
\usepackage[round]{natbib}

\author[Jiafu He et al.]{Jiafu He
  \and Haiyu Zeng
  \and Yanbo Zhang}
\title{Ramsey goodness of stars and fans for the Haj\'os graph}
\affiliation{
  School of Mathematical Sciences, Hebei Normal University, Shijiazhuang, China\\
  Hebei Research Center of the Basic Discipline Pure Mathematics, Shijiazhuang, China}
\begin{document}
\publicationdata
{vol.27:3}
{2025}
{11}
{10.46298/dmtcs.15817}
{2025-06-10; None}
{2025-09-08}

\maketitle
\begin{abstract}
  Given two graphs $G_1$ and $G_2$, the Ramsey number $R(G_1,G_2)$ denotes the smallest integer $N$ such that any red-blue coloring of the edges of $K_N$ contains either a red $G_1$ or a blue $G_2$. Let $G_1$ be a graph with chromatic number $\chi$ and chromatic surplus $s$, and let $G_2$ be a connected graph with $n$ vertices. 
  The graph $G_2$ is said to be Ramsey-good for the graph $G_1$ (or simply $G_1$-good) if, for $n \ge s$,
  \[R(G_1,G_2)=(\chi-1)(n-1)+s.\]
  
  The $G_1$-good property has been extensively studied for star-like graphs when $G_1$ is a graph with $\chi(G_1)\ge 3$, as seen in works by Burr-Faudree-Rousseau-Schelp (\textit{J. Graph Theory}, 1983), Li-Rousseau (\textit{J. Graph Theory}, 1996), Lin-Li-Dong (\textit{European J. Combin.}, 2010), Fox-He-Wigderson (\textit{Adv. Combin.}, 2023), and Liu-Li (\textit{J. Graph Theory}, 2025), among others. However, all prior results require $G_1$ to have chromatic surplus~$1$. In this paper, we extend this investigation to graphs with chromatic surplus 2 by considering the Haj\'os graph $H_a$. For a star $K_{1,n}$, we prove that $K_{1,n}$ is $H_a$-good if and only if $n$ is even. For a fan $F_n$ with $n\ge 111$, we prove that $F_n$ is $H_a$-good.
\end{abstract}

\section{Introduction}
The Ramsey number is a bivariate function that assigns a positive integer $R(G_1,G_2)$ to every pair of simple graphs $G_1$ and $G_2$. It is defined as the smallest integer $N$ such that any graph $\Gamma$ on $N$ vertices contains $G_1$ as a subgraph, or its complement $\overline{\Gamma}$ contains $G_2$ as a subgraph.

The study of Ramsey numbers for sparse graphs has flourished since the 1970s. At that time, \cite{Burr1981} established a general lower bound for the Ramsey number of any pair of graphs. Suppose $G_2$ is a connected graph, and let $\chi(G_1)$ and $s(G_1)$ denote the chromatic number and the chromatic surplus of $G_1$, respectively, where the chromatic surplus refers to the smallest size of a color class over all proper $\chi(G_1)$-colorings of $G_1$. Then, if $|G_2|\ge s(G_1)$,
\begin{equation}\label{equal:basic}
	R(G_1,G_2)\ge (\chi(G_1)-1)(|G_2|-1)+s(G_1).
\end{equation}
When equality holds in~\eqref{equal:basic}, the graph $G_2$ is said to be \emph{$G_1$-good}, or equivalently, $G_2$ is \emph{Ramsey-good} for $G_1$. In the special case where $G_1$ is the complete graph $K_k$, such graphs $G_2$ are referred to as \emph{$k$-good} by \cite{BurrErdos1983}.

When $G_1$ is the complete $(k+1)$-partite graph $K_{1,m_1,m_2,\ldots,m_k}$ and $G_2$ is a star, \cite{Burr1983} proved that for $m=\min\{m_i\mid i\in [k]\}$ and sufficiently large $n$,
\[R(K_{1,m_1,m_2, \ldots,m_k},K_{1,n})=\begin{cases}
	k \cdot (n+m-2)+1 & \text{if both } m \text{ and } n \text{ are even}, \\
	k \cdot (n+m-1)+1 & \text{otherwise}.
\end{cases}\]
It follows that $K_{1,n}$ is $K_{1,m_1,m_2,\ldots,m_k}$-good only when $m_1=1$, or when $m_1=2$ and $n$ is even.

A frequently studied class of star-like graphs is the fan. A fan $F_n$ is the graph formed by $n$ triangles sharing a common vertex; see Figure~\ref{Fig:HajosFan}. The concept of the fan was first introduced into extremal graph theory by \cite{Erdos1995}, who studied its extremal graph and Tur\'an number. For results on the Ramsey numbers of fans, see~\cite{Chen2021,Dvorak2023,Lin2009,Zhang2015}.

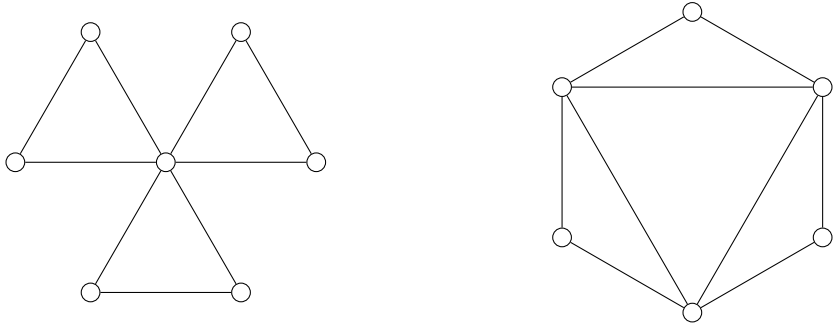
\begin{figure}[ht]
	\centering
	\begin{tikzpicture}
		
		\begin{scope}[xshift=-3.5cm]
			\foreach \i in {1,2,3,4,5,6} {
				\node[draw, circle, fill=white, inner sep=2.5pt] (f\i) at (120-60*\i:2cm) {};
			}
			
			\node[draw, circle, fill=white, inner sep=2.5pt] (f7) at (0,0) {};
			
			\draw (f1) -- (f2);
			\draw (f3) -- (f4);
			\draw (f5) -- (f6);
			
			\foreach \i in {1,2,3,4,5,6} {
				\draw (f7) -- (f\i);
			}
		\end{scope}
		
		\begin{scope}[xshift=3.5cm]
			\foreach \i in {1,2,3,4,5,6} {
				\node[draw, circle, fill=white, inner sep=2.5pt] (h\i) at (90-60*\i:2cm) {};
			}
			
			\foreach \i in {1,2,3,4,5,6} {
				\pgfmathtruncatemacro{\j}{mod(\i,6)+1}
				\draw (h\i) -- (h\j);
			}
			
			\draw (h1) -- (h3);
			\draw (h3) -- (h5);
			\draw (h1) -- (h5);
		\end{scope}
		
	\end{tikzpicture}
	\caption{The fan $F_3$ (left) and the Haj\'os graph (right)}
	\label{Fig:HajosFan}
\end{figure}

The concept of a fan can be extended to that of a generalized fan. Given a graph $H$, let $nH$ denote the disjoint union of $n$ copies of $H$, and let $H_1+H_2$ denote the join of two disjoint graphs $H_1$ and $H_2$, obtained by adding all edges between $V(H_1)$ and $V(H_2)$. The graph $K_1+nH$ is called a generalized fan. In particular, if $H$ is $K_1$, then $K_1+nH$ is the star $K_{1,n}$; if $H$ is $K_2$, then $K_1+nH$ is the fan $F_n$.

\cite{Li1996} investigated the Ramsey-goodness of generalized fans. For any graphs $H$ and $G$, they proved that $K_1+nH$ is Ramsey-good for $K_2+G$ when $n$ is sufficiently large.

Observe that in any proper coloring of the graph $K_2+G$, there are two color classes containing only one vertex each. What happens if we require only one color class to have a single vertex? Let $K_1+K_k(m)$ denote the complete $(k+1)$-partite graph in which one partite set has size $1$ and each of the remaining $k$ partite sets has size $m$. \cite{Lin2010} showed that for $k\ge 2$, if $m$ is odd or if $n|H|$ is odd, then
\[
R(K_1+K_k(m),K_1+nH)=k(n|H|+m-1)+1.
\]
It follows that whenever $m\ge 2$, the graph $K_1+nH$ is not Ramsey-good for $K_1+K_k(m)$ in this setting.

Subsequently, without relying on the Erd\H{o}s-Stone-Simonovits stability theorem, \cite{Chung2022} determined the exact value of $R(K_{1,m_1,m_2,\ldots,m_k},K_1+nH)$ for large $n$ in all cases. Let $m=\min\{m_i\mid i\in [k]\}$. Then
\[
r(K_{1,m_1,m_2, \ldots,m_k}, K_1+nH)=
\begin{cases}
	k \cdot (n|H|+m-2)+1 & \text{if both } m \text{ and } n|H| \text{ are even}, \\
	k \cdot (n|H|+m-1)+1 & \text{otherwise}.
\end{cases}
\]

Another frequently studied star-like graph is the book $B_{k,n}$, which is the graph $K_k+(n-k)K_1$. Equivalently, $B_{k,n}$ can be viewed as the graph obtained by blowing up the center of the star $K_{1,n-k}$ into a clique $K_k$, while preserving the adjacency between each vertex of $K_k$ and the remaining vertices. For results concerning the Ramsey goodness of $B_{k,n}$, we refer the reader to~\cite{Fox2023,Nikiforov2004,Liu2025}. For results on the Ramsey non-goodness of $B_{k,n}$, see~\cite{Fan2023,Fan2024,Lin2021}.

It is easy to observe that in all the aforementioned works, when $G_2$ is a star or a star-like graph, the corresponding graph $G_1$ is always required to have chromatic surplus $1$. We extend these results by considering $G_1$ to be the Haj\'os graph $H_a$, which has chromatic surplus $2$. The Haj\'os graph, named after the Hungarian mathematician Gy\"orgy Haj\'os, is a graph consisting of six vertices and nine edges. It is constructed by starting with a triangle and, for each of its edges, adding a new vertex and joining it to both endpoints of that edge; see Figure~\ref{Fig:HajosFan}.

In this paper, we establish two main results concerning Ramsey goodness for the Haj\'os graph.

Our first result establishes that the star $K_{1,n}$ is $H_a$-good if and only if $n$ is even.

\begin{theorem}\label{thm:star}
	$R(H_a,K_{1,n})=\begin{cases}
		2n+2 & \text{for even } n\ge 2, \\
		2n+3 & \text{for odd } n\ge 3.
	\end{cases}$
\end{theorem}

Our second result determines that $F_n$ is $H_a$-good for $n \geq 111$.

\begin{theorem}\label{thm:fan}
	$R(H_a,F_n)=4n+2$ for $n\ge 111$.
\end{theorem}

We present the proofs of the two theorems in Section~\ref{section2} and Section~\ref{section3}, respectively. The proof of Theorem~\ref{thm:fan} is more involved, and the lower bound $n\ge 111$ is not the best possible; rather, it arises from the limitations of our method. Therefore, identifying a smaller positive integer $n$ such that $F_n$ is $H_a$-good is a worthwhile direction for further research.

At the end of this section, we introduce some concepts and notation that will be used throughout. We use $[n]$ to denote the set $\{1,2,\ldots,n\}$. For a graph $G$ and a vertex subset $X\subseteq V(G)$, the subgraph induced by $X$ is denoted by $G[X]$, whose vertex set is $X$ and whose edge set consists of all edges in $G$ with both endpoints in $X$. For $u\in V(G)$ and $X\subseteq V(G)$, let $N_X(u)$ denote the set of neighbors of $u$ in $X$, and let $|N_X(u)|$ denote the number of such neighbors, also written as $d_X(u)$. When $X=V(G)$, we write $N_X(u)$ and $d_X(u)$ simply as $N(u)$ and $d(u)$, respectively. The minimum and maximum degrees of $G$ are denoted by $\delta(G)$ and $\Delta(G)$, respectively. Notation and terminology not explicitly defined follow \cite{Bondy2008}.

\section{Proof of Theorem~\ref{thm:star}}\label{section2}

When $n$ is even, the lower bound can be derived from inequality~(\ref{equal:basic}):
\[R(H_a,K_{1,n})\ge (\chi(H_a)-1)(|K_{1,n}|-1)+s(H_a)=2n+2.\]

When $n$ is odd, let $\ell=(n+1)/2$. Consider the graph $(\ell K_2)+(\ell K_2)$, which consists of two disjoint copies of a matching with $\ell$ edges, where the edges between the two copies form a complete bipartite graph. It is straightforward to verify that this graph does not contain the Haj\'os graph as a subgraph, and that the maximum degree in its complement is $n-1$. Therefore, $R(H_a,K_{1,n})\ge 2n+3$ for odd $n$.

For the upper bound, we first consider the case $n=2$, namely, proving that $R(H_a,K_{1,2})\le 6$. For any graph $H$ of order $6$, if its complement $\overline{H}$ does not contain $K_{1,2}$ as a subgraph, then the edge set of $\overline{H}$ forms a matching. Consequently, $H$ must contain $K_{2,2,2}$ as a subgraph, which in turn contains $H_a$ as a subgraph.

Now suppose $n\ge 3$. Let $G$ be an arbitrary graph with $2n+2+\mathbf{1}_{\text{odd}}(n)$ vertices, where $\mathbf{1}_{\text{odd}}(n)$ is an indicator function: $\mathbf{1}_{\text{odd}}(n)=0$ if $n$ is even, and $\mathbf{1}_{\text{odd}}(n)=1$ if $n$ is odd. Suppose $\overline{G}$ contains no copy of $K_{1,n}$. We now prove that $G$ must contain the Haj\'os graph as a subgraph.

Since $\Delta(\overline{G})\le n-1$, it follows that $\delta(G)\ge n+2+\mathbf{1}_{\text{odd}}(n)$. We proceed by considering three cases.

\setcounter{case}{0}
\begin{case}
	The graph $G$ does not contain $K_4$ as a subgraph.
\end{case}

By the classical result of \cite{Chvatal1977}, we have $R(K_3,K_{1,n})=2n+1$. Therefore, the graph $G$ must contain a triangle; without loss of generality, let its vertices be $u_1,u_2,u_3$. For $1\le i<j\le 3$, the minimum degree condition implies
\[
d(u_i)+d(u_j)\ge 2\delta(G)\ge 2n+4+2\times \mathbf{1}_{\text{odd}}(n)\ge |G|+2.
\]
Thus, vertices $u_1$ and $u_2$ share a common neighbor distinct from $u_3$, denoted $u_4$; similarly, $u_1$ and $u_3$ share a common neighbor $u_5\ne u_2$, and $u_2$ and $u_3$ share a common neighbor $u_6\ne u_1$. Since $G$ does not contain $K_4$ as a subgraph, the vertices $u_4,u_5,u_6$ must be pairwise distinct. Consequently, these six vertices induce a Haj\'os graph.

\begin{case} 
	The graph $G$ contains $K_5-e$ as a subgraph.
\end{case}

Let us denote the five vertices of one such $K_5-e$ by $v_1,v_2,v_3,v_4,v_5$, and let the set of remaining vertices be $W$. Then we have $|W|=2n-3+\mathbf{1}_{\text{odd}}(n)$. Note that $K_5-e$ contains a clique of four vertices; without loss of generality, assume that $v_1,v_2,v_3,v_4$ induce a $K_4$. By the minimum degree condition, for each $i\in[4]$, the vertex $v_i$ has at least $n-2+\mathbf{1}_{\text{odd}}(n)$ neighbors in $W$. Since $n\ge 3$, it follows that
\[4(n-2+\mathbf{1}_{\text{odd}}(n))>2n-3+\mathbf{1}_{\text{odd}}(n)=|W|.\]
Therefore, among $v_1,v_2,v_3,v_4$, there must exist two vertices, say $v_1$ and $v_2$, which share a common neighbor in $W$, denoted by $v_6$. Because the vertex $v_5$ has at least two neighbors in $\{v_2,v_3,v_4\}$, we may assume without loss of generality that $v_5$ is adjacent to both $v_2$ and $v_3$. Thus, we obtain a Haj\'os graph, where the vertex subsets $\{v_1,v_2,v_3\}$, $\{v_1,v_2,v_6\}$, $\{v_2,v_3,v_5\}$, and $\{v_1,v_3,v_4\}$ each induce a triangle.

\begin{case}
	The graph $G$ does not contain $K_5-e$ as a subgraph, but contains $K_4$ as a subgraph.
\end{case}

Let $w_1,w_2,w_3,w_4$ be the vertices of a $K_4$ in $G$.

We first consider the case when $n$ is even. Since $\delta(G)\ge n+2$, each vertex $w_i$ for $i\in [4]$ has at least $n-1$ neighbors in $V(G)\setminus \{w_1,w_2,w_3,w_4\}$. As $|V(G)\setminus \{w_1,w_2,w_3,w_4\}|=2n-2$, there must exist two vertices, say $w_1$ and $w_2$, that share a common neighbor in $V(G)\setminus \{w_1,w_2,w_3,w_4\}$, denoted by $w_5$. Since $G$ does not contain $K_5-e$ as a subgraph, neither $w_3$ nor $w_4$ is adjacent to $w_5$.

Let $X=V(G)\setminus \{w_1,w_2,w_3,w_4,w_5\}$. Then $|X|=2n-3$. By the minimum degree condition,
\[d_X(w_1)\ge n-2, d_X(w_2)\ge n-2, d_X(w_3)\ge n-1, \text{and } d_X(w_4)\ge n-1.\]
If $w_1$ and $w_3$ have a common neighbor in $X$, denoted by $w_6$, then we can construct a Haj\'os graph in which the vertex subsets $\{w_1,w_2,w_3\}$, $\{w_1,w_2,w_5\}$, $\{w_2,w_3,w_4\}$, and $\{w_1,w_3,w_6\}$ each induce a triangle. Hence, $N_X(w_1)\cap N_X(w_3)=\emptyset$. By symmetry, we have $N_X(w_2)\cap N_X(w_3)=\emptyset$.

Since $|X|=2n-3$, $d_X(w_1)\ge n-2$ and $d_X(w_3)\ge n-1$, the sets $N_X(w_1)$ and $N_X(w_3)$ form a partition of $X$, and we have $|N_X(w_1)|=n-2$ and $|N_X(w_3)|=n-1$. Given that $d_X(w_2)\ge n-2$ and $N_X(w_2)\cap N_X(w_3)=\emptyset$, it follows that $N_X(w_1)=N_X(w_2)$. Therefore,
\[
N(w_1)\cap N(w_2)=N_X(w_1)\cup \{w_3,w_4,w_5\} \text{ and } |N(w_1)\cap N(w_2)|=n+1.
\]
Since $n$ is even, by the classical result of \cite{Harary1972},
\[
R(K_{1,2},K_{1,n})=n+1=|N(w_1)\cap N(w_2)|.
\]
By assumption, $\overline{G}$ does not contain $K_{1,n}$ as a subgraph. Hence, the graph $K_{1,2}$ appears within the common neighborhood of $w_1$ and $w_2$, and together with $w_1$ and $w_2$ induces a subgraph isomorphic to $K_5-e$. This contradicts the assumption that $G$ does not contain $K_5-e$ as a subgraph.

Now consider the case where $n$ is odd. Let $Y=V(G)\setminus \{w_1,w_2,w_3,w_4\}$. Since $\delta(G)\ge n+3$, each vertex $w_i$ for $i\in [4]$ has at least $n$ neighbors in $Y$. Given that $|Y|=2n-1$, any two vertices among $w_1,w_2,w_3,w_4$ share at least one common neighbor in $Y$. Without loss of generality, assume that $w_1$ and $w_2$ share a common neighbor $w_5$ in $Y$, and $w_2$ and $w_3$ share a common neighbor $w_6$ in $Y$. As $G$ does not contain $K_5-e$ as a subgraph, vertices $w_5$ and $w_6$ must be distinct. Consequently, we obtain a Haj\'os graph in which the vertex subsets $\{w_1,w_2,w_3\}$, $\{w_1,w_2,w_5\}$, $\{w_2,w_3,w_6\}$, and $\{w_1,w_3,w_4\}$ each induce a triangle. \qed

\section{Proof of Theorem~\ref{thm:fan}}\label{section3}
The proof of Theorem~\ref{thm:fan} relies on the following key lemma, which was established by \cite{Zeng2025}.
\begin{lemma}\label{lem:W4}
	$R(W_4,F_n)=4n+1$ for $n\ge 111$.
\end{lemma}

Now we begin the proof of Theorem~\ref{thm:fan}.

\begin{proof}[of Theorem~\ref{thm:fan}]
	The lower bound follows directly from inequality~(\ref{equal:basic}):
	\[
	R(H_a,F_n)\ge (\chi(H_a)-1)(|F_n|-1)+s(H_a)=(3-1)(2n+1-1)+2=4n+2.
	\]  
	To prove the upper bound, we proceed by contradiction. Assume that a graph $G$ with $4n+2$ vertices contains neither $H_a$ as a subgraph nor $F_n$ as a subgraph in $\overline{G}$. We will derive contradictions in the following two cases.
	
	\setcounter{case}{0}
	\begin{case}
		$\Delta(\overline{G})\ge 2n+2$
	\end{case}
	
	Let $u$ be a vertex in $\overline{G}$ with the maximum degree. Denote by $H^{'}$ the subgraph of $\overline{G}$ induced by the set of vertices adjacent to $u$ in $\overline{G}$, i.e., $H^{'}=\overline{G}[N_{\overline{G}}(u)]$. Let $M$ be a maximum matching in $H^{'}$. Then $M$ must contain exactly $n-1$ edges. This is because if $M$ has at least $n$ edges, these $n$ edges together with $u$ would form an $F_n$ in $\overline{G}$, a contradiction. Conversely, if $M$ has at most $n-2$ edges, the number of vertices in $H^{'}$ not covered by $M$ (i.e., the remaining vertices after removing the vertices of $M$) is at least $2n+2-2(n-2)=6$. These vertices form a complete graph in $G$, which clearly contains $H_a$ as a subgraph, leading to a contradiction.
	
	Since $\Delta(\overline{G})\ge 2n+2$, it follows that $|V(H^{'})\setminus V(M)|\ge4$, and the vertices in $V(H^{'})\setminus V(M)$ induce a complete subgraph in $G$.
	
	- If $|V(H^{'})\setminus V(M)|\ge 6$, it is evident that $G$ contains $H_a$ as a subgraph.
	
	- If $|V(H^{'})\setminus V(M)|=5$, consider any edge $y_1y_2\in M$. At least one of $y_1$ or $y_2$ is adjacent to at least four vertices in $V(H^{'})\setminus V(M)$ in $G$. Otherwise, $H^{'}$ would contain a matching larger than $M$, contradicting the maximality of $M$. Without loss of generality, assume $y_1$ is adjacent to at least four vertices in $V(H^{'})\setminus V(M)$. Then the subgraph of $G$ induced by $(V(H^{'})\setminus V(M))\cup\{y_1\}$ must contain $H_a$ as a subgraph, leading to a contradiction.
	
	- If $|V(H^{'})\setminus V(M)|=4$, consider any two edges $y_1y_2$ and $y_3y_4\in M$. For each edge, at least one endpoint must be adjacent to at least three vertices in $V(H^{'})\setminus V(M)$ in $G$. Otherwise, $H^{'}$ would contain a matching larger than $M$, again contradicting the maximality of $M$. Without loss of generality, assume $y_1$ and $y_3$ are each adjacent to at least three vertices in $V(H^{'})\setminus V(M)$. Then the subgraph of $G$ induced by $(V(H^{'})\setminus V(M))\cup\{y_1,y_3\}$ must contain $H_a$ as a subgraph, leading to a contradiction.

	\begin{case}
		$\delta(G)\ge 2n$
	\end{case}
	
	According to Lemma~\ref{lem:W4}, the graph $G$ contains $W_4$ as a subgraph. Let the center of this $W_4$ be $u_0$, with the corresponding cycle $u_1u_2u_3u_4u_1$. We have the following claim.
	
	\begin{claim}
		The edges $u_1u_3, u_2u_4\not\in E(G)$.
	\end{claim}
	\begin{proof}
		We only need to prove that $u_1u_3\not\in E(G)$; the other case can be proven similarly. Suppose $u_1u_3\in E(G)$. Since $\delta(G)\ge 2n$, the set $\{u_1, u_2, u_3\}$ must contain two vertices, say $u_1$ and $u_3$, that have a common neighbor in $V(G)\setminus V(W_4)$. Otherwise, this would contradict the fact that $G$ has $4n+2$ vertices. Let $u_5$ be a common neighbor of $u_1$ and $u_3$ in $V(G)\setminus V(W_4)$. Then the subgraph induced by the vertex set $\{u_0, u_1, u_2, u_3, u_4, u_5\}$ in $G$ contains $H_a$ as a subgraph, a contradiction.
	\end{proof}
	
	If $u_1$ and $u_2$ have a common neighbor distinct from $u_0$, denoted by $u_5$, then the subgraph induced by the vertex set $\{u_0, u_1, u_2, u_3, u_4, u_5\}$ contains $H_a$ as a subgraph. Therefore, the only common neighbor of $u_1$ and $u_2$ can be $u_0$. Furthermore, since $\delta(G)\ge 2n$, $u_2$ must have at least $2n-1$ neighbors in $V(G)\setminus N(u_1)$. Similarly, $u_4$ must also have at least $2n-1$ neighbors in $V(G)\setminus N(u_1)$. Hence, the number of common neighbors of $u_2$ and $u_4$ in $V(G)\setminus N(u_1)$ is at least
	\[
	2(2n-1)-|V(G)\setminus N(u_1)|\ge 2(2n-1)-(4n+2-2n)=2n-4.
	\]
	We denote the set of common neighbors of vertices $u_2$ and $u_4$, excluding $u_0$, as $U_1$. Then, $|U_1|\ge 2n-4$. Similarly, let $U_2$ represent the set of common neighbors of vertices $u_1$ and $u_3$, excluding $u_0$. By symmetry, we have $|U_2|\ge 2n-4$. Since $u_1$ and $u_2$ have only one common neighbor, $u_0$, it follows that $U_1\cap U_2=\emptyset$. It is easy to verify that $u_1, u_3\in U_1$ and $u_2, u_4\in U_2$. To avoid the appearance of the graph $H_a$, for $i=1,3$ and $j=2,4$, the only common neighbor of $u_i$ and $u_j$ must be $u_0$. Therefore, the vertices $u_1$ and $u_3$ are isolated in the graph $G[U_1]$, and the vertices $u_2$ and $u_4$ are isolated in the graph $G[U_2]$. Furthermore, we have the following claim.
	
	\begin{claim}\label{star}
		Every connected component of the graph $G[U_1\setminus \{u_1,u_3\}]$ is either an isolated vertex or a star.
	\end{claim}
	\begin{proof}
		First, we prove that the graph $G[U_1\setminus \{u_1,u_3\}]$ does not contain $P_4$ as a subgraph. We proceed by contradiction. Suppose $G[U_1\setminus \{u_1,u_3\}]$ contains a $P_4$ as a subgraph, with vertices $x_1, x_2, x_3, x_4$. The subgraph induced by the vertex set $\{u_2, u_4, x_1, x_2, x_3, x_4\}$ in $G$ contains $H_a$ as a subgraph, which leads to a contradiction. Thus, each connected component of the graph $G[U_1\setminus \{u_1,u_3\}]$ is either an isolated vertex, a star, or a triangle. Note that when a connected component is $K_2$, it is also considered a star.
		
		Next, we only need to prove that the graph $G[U_1\setminus \{u_1,u_3\}]$ does not contain a triangle as a subgraph. Again, we proceed by contradiction. Suppose $G[U_1\setminus \{u_1,u_3\}]$ contains a triangle with vertices $x_1, x_2, x_3$. There must exist two vertices among $\{x_1, x_2, x_3\}$, say $x_1$ and $x_3$, that have a common neighbor in $V(G)\setminus \{x_1, x_2, x_3, u_2, u_4\}$, otherwise this would contradict the fact that the number of vertices in $G$ is $4n+2$. Let $x_4$ be a common neighbor of $x_1$ and $x_3$ in $V(G)\setminus \{x_1, x_2, x_3, u_2, u_4\}$. Then the subgraph induced by the vertex set $\{x_1, x_2, x_3, x_4, u_2, u_4\}$ in $G$ contains $H_a$ as a subgraph, which leads to a contradiction.
	\end{proof}
	
	By symmetry, it can also be shown that each connected component of the graph $G[U_2\setminus \{u_2,u_4\}]$ is either an isolated vertex or a star. Hence, the following claim holds.
	
	\begin{claim}\label{Uupper}
		$|U_1|\le 2n$ and $|U_2|\le 2n$.
	\end{claim}
	\begin{proof}
		We proceed by contradiction. Assume $|U_1|\ge 2n+1$ or $|U_2|\ge 2n+1$. By symmetry, we may assume $|U_1|\ge 2n+1$. It suffices to find a fan $F_n$ centered at $u_1$ in $\overline{G}$, leading to a contradiction.
		
		By the previous claim, each connected component of the graph $G[U_1\setminus \{u_1\}]$ is either an isolated vertex or a star. Let the nontrivial (non-isolated) connected components of this graph be $C_1,\ldots,C_k$. Note that in each connected component, at most one vertex has degree greater than $1$ in $G[U_1\setminus \{u_1\}]$.
		
		When $k\ge 2$, for each $1\le i\le k-1$, the center vertex of $C_i$ and a leaf vertex of $C_{i+1}$ form an edge in $\overline{G}$. Similarly, the center vertex of $C_k$ and a leaf vertex of $C_1$ form an edge in $\overline{G}$. These $k$ edges form a matching in $\overline{G}$. The remaining vertices in $\overline{G}[U_1\setminus \{u_1\}]$ induce a complete subgraph, allowing us to find a matching of $n$ edges in $\overline{G}[U_1\setminus\{u_1\}]$. This gives a fan $F_n$ centered at $u_1$ in $\overline{G}$, a contradiction.
		
		When $k=1$, the graph $G[U_1\setminus \{u_1\}]$ has only one nontrivial connected component, which must be a star. Let $v_0$ be the center of this star. Then $v_0u_3\in E(\overline{G})$. Since the remaining vertices in $\overline{G}[U_1\setminus \{u_1\}]$ induce a complete subgraph, $v_0u_3$ can be extended to a matching of $n$ edges in $\overline{G}[U_1\setminus \{u_1\}]$. This gives a fan $F_n$ centered at $u_1$ in $\overline{G}$, a contradiction.
		
		When $k=0$, the graph $G[U_1\setminus \{u_1\}]$ is an empty graph, so $\overline{G}[U_1\setminus \{u_1\}]$ contains a matching of $n$ edges. This also gives a fan $F_n$ centered at $u_1$ in $\overline{G}$, a contradiction.
	\end{proof}
	
	We denote the set of vertices outside $U_1\cup U_2\cup \{u_0\}$ by $W$. Then, 
	\[
	|W|\le 4n+2-2(2n-4)-1=9.
	\]
	Furthermore, we partition $W$ into four subsets: $W_1$, $W_2$, $W_3$, and $W_4$. If a vertex in $W$ is adjacent to $u_2$ or $u_4$, it belongs to $W_1$. If a vertex in $W$ is adjacent to $u_1$ or $u_3$, it belongs to $W_2$. If a vertex $w\in W$ is not adjacent to any of $u_1$, $u_2$, $u_3$, or $u_4$, and the number of its neighbors in $U_2$ is at least the number of its neighbors in $U_1$, i.e., $d_{U_2}(w)\ge d_{U_1}(w)$, then it belongs to $W_3$. Otherwise, if $d_{U_2}(w)<d_{U_1}(w)$, it belongs to $W_4$. To avoid the appearance of the graph $H_a$, we have $W_1\cap W_2=\emptyset$. Consequently, $W_1$, $W_2$, $W_3$, and $W_4$ form a partition of $W$.
	
	Note that the neighbors of $u_2$ must belong to $W_1\cup U_1\cup \{u_0\}$. Since $\delta(G)\ge 2n$, we have $|W_1|\ge 2n-1-|U_1|$. By symmetry, $|W_2|\ge 2n-1-|U_2|$. Hence, $|W_3|+|W_4|\le (4n+2)-1-2(2n-1)=3$.
	
	By the pigeonhole principle, either $|U_1|+|W_1|+|W_3|\ge 2n+1$, or $|U_2|+|W_2|+|W_4|\ge 2n+1$. Without loss of generality, assume the former holds. Next, we will show that if $G$ does not contain $H_a$ as a subgraph, then $\overline{G}[U_1\cup W_1\cup W_3]$ must contain a fan $F_n$ centered at $u_1$.
	
	Note that every vertex in $(U_1\cup W_1\cup W_3)\setminus \{u_1,u_3\}$ is not adjacent to $u_1$ or $u_3$. Therefore, it suffices to find a matching of $n$ edges in $\overline{G}[(U_1\cup W_1\cup W_3)\setminus \{u_1\}]$, which yields a fan $F_n$ centered at $u_1$ in $\overline{G}$.
	
	From $W_1\cup W_3$, select $2n+1-|U_1|$ vertices, denoted by $w_1,\ldots,w_t$, where $t=2n+1-|U_1|$. Since $|U_1|\ge 2n-4$, it follows that $t\le 5$. We first find a matching $M_1$ in $\overline{G}[\{w_1,\ldots,w_t\}\cup(U_1\setminus\{u_1,u_3\})]$ such that each edge of $M_1$ has at least one endpoint in $\{w_1,\ldots,w_t\}$ and $V(M_1)$ includes as many vertices as possible from $\{w_1,\ldots,w_t\}$. The following claim holds:
	
	\begin{claim}\label{Usize}
		At most one vertex in $\{w_1,\ldots,w_t\}$ does not belong to $V(M_1)$.
	\end{claim}
	\begin{proof}
		Suppose at least three vertices in $\{w_1,\ldots,w_t\}$ are not in $V(M_1)$. Without loss of generality, assume $w_1,w_2,w_3\notin V(M_1)$. It follows that $w_1,w_2,w_3$ form a triangle in $G$. To see this, if any edge is missing from $G$, say $w_1w_2\in E(\overline{G})$, we could add the edge $w_1w_2$ to $M_1$, forming a new matching that includes more vertices from $\{w_1,\ldots,w_t\}$, contradicting the choice of $M_1$. Additionally, by the choice of $M_1$, $w_1$, $w_2$, and $w_3$ must each be adjacent to every vertex in $U_1\setminus(\{u_1,u_3\}\cup V(M_1))$. Since $t\le 5$, at most two vertices of $V(M_1)$ are in $U_1$. Thus, $|U_1\setminus(\{u_1,u_3\}\cup V(M_1))|\ge (2n-4)-(2+2)=2n-8$. Selecting any three vertices from $U_1\setminus(\{u_1,u_3\}\cup V(M_1))$, together with $w_1,w_2,w_3$, induces a complete multipartite graph $K_{1,1,1,3}$ in $G$, which contains $H_a$ as a subgraph, a contradiction.
		
		If exactly two vertices in $\{w_1,\ldots,w_t\}$ are not in $V(M_1)$, assume $w_1,w_2\notin V(M_1)$. Then $w_1w_2\in E(G)$ by the choice of $M_1$. Moreover, $w_1$ and $w_2$ must be adjacent to every vertex in $U_1\setminus(\{u_1,u_3\}\cup V(M_1))$. Since at most three vertices in $V(M_1)$ are in $U_1$, we have $|U_1\setminus(\{u_1,u_3\}\cup V(M_1))|\ge |U_1|-5\ge 2n-9$.
		
		First, consider $w_1,w_2\in W_1$. If $w_1$ and $w_2$ are both adjacent to either $u_2$ or $u_4$, without loss of generality, assume $w_1u_2,w_2u_2\in E(G)$. Selecting any three vertices from $U_1\setminus(\{u_1,u_3\}\cup V(M_1))$, together with $w_1$, $w_2$, and $u_2$, induces a complete multipartite graph $K_{1,1,1,3}$ in $G$, which contains $H_a$ as a subgraph, leading to a contradiction. If $w_1$ and $w_2$ are adjacent to different vertices in $\{u_2,u_4\}$, assume $w_1u_2,w_2u_4\in E(G)$. Selecting any two vertices $u_x,u_y$ from $U_1\setminus(\{u_1,u_3\}\cup V(M_1))$, the four sets $\{w_1,w_2,u_x\}$, $\{w_1,w_2,u_y\}$, $\{w_1,u_2,u_y\}$, and $\{w_2,u_4,u_y\}$ each induce a triangle. Thus, $G$ contains $H_a$ as a subgraph, a contradiction.
		
		Next, consider $w_1,w_2\in W_3$. By the definition of $W_3$, $w_1$ and $w_2$ each have at least $2n-9$ neighbors in $U_2$. Thus, $w_1$ and $w_2$ share at least $2(2n-9)-|U_2|\ge 2n-18$ common neighbors in $U_2$, where this inequality follows from Claim~\ref{Uupper}. 
		
		If each common neighbor of $w_1$ and $w_2$ in $U_2$ has degree at least $2$ in $G[U_2]$, note that each connected component in $G[U_2]$ has at most one vertex of degree at least $2$. Hence, $G[U_2]$ must have at least $2n-18$ components, each containing a vertex of degree at least $2$, contradicting $|U_2|\le 2n$. Thus, there exists a common neighbor $u_z$ of $w_1$ and $w_2$ in $U_2$ with degree at most $1$ in $G[U_2]$. Then the vertex $u_z$ has at least $2n-2-|W|\ge 2n-11$ neighbors in $U_1$. Consequently, $u_z$, $w_1$, and $w_2$ share at least $(2n-11)+(|U_1|-5)-|U_1|\ge 2n-16$ common neighbors in $U_1$. Selecting three such common neighbors together with $u_z,w_1,w_2$ induces a complete multipartite graph $K_{1,1,1,3}$, which contains $H_a$ as a subgraph, a contradiction.
		
		If $w_1$ and $w_2$ belong to $W_1$ and $W_3$, respectively, let $w_1\in W_1$ and $w_2\in W_3$. Then $w_1$ must be adjacent to either $u_2$ or $u_4$. By symmetry, assume $w_1u_2\in E(G)$. According to the selection of $W_3$, $w_2$ has at least $2n-9$ neighbors in $U_2$. Therefore, $w_2$ must have a neighbor $u_z$ in $U_2$ whose degree in $G[U_2]$ is at most 1; otherwise, this would contradict $|U_2|\le 2n$. Consequently, the number of neighbors of $u_z$ in $U_1$ is at least $2n-2-|W|\ge 2n-11$. Thus, the number of common neighbors of $u_z$, $w_1$, and $w_2$ in $U_1$ is at least $(2n-11)+(|U_1|-5)-|U_1|\ge 2n-16$. Selecting two common neighbors, together with $u_z$, $w_1$, $w_2$, and $u_2$, forms a subgraph in $G$ that contains $H_a$ as a subgraph, leading to a contradiction.
	\end{proof}
	
	Based on Claim~\ref{star}, every connected component in the graph $G[U_1\setminus (\{u_1,u_3\}\cup V(M_1))]$ is either an isolated vertex or a star. Denote the connected components of this graph as $D_1,\ldots,D_k$. Note that each connected component contains at most one vertex with degree greater than $1$ in $G[U_1\setminus (\{u_1,u_3\}\cup V(M_1))]$. Let $H=G[(U_1\cup \{w_1,\ldots,w_t\})\setminus (\{u_1\}\cup V(M_1))]$. We establish the following claim.
	
	\begin{claim}
		There exists a matching $M_2$ in $\overline{H}$ that covers all vertices in $H$ with degree greater than $1$.
	\end{claim}
	
	\begin{proof}
		When $k\ge 2$, for $i\in [k]$ and $D_{k+1}:=D_1$, if $D_i$ contains a vertex with degree greater than $1$, then this vertex has no edge to a vertex with degree at most $1$ in $D_{i+1}$. That is, there exists an edge between these two vertices in $\overline{H}$. Add this edge to $M_2$. If exactly one vertex in $\{w_1,\ldots,w_t\}$ is not in $V(M_1)$, assume without loss of generality that $w_1\notin V(M_1)$. Then $w_1u_3\in E(\overline{G})$, and we add the edge $w_1u_3$ to $M_2$. If $\{w_1,\ldots,w_t\}\subseteq V(M_1)$, the edge $w_1u_3$ does not need to be included in $M_2$. Clearly, $M_2$ covers all vertices in $H$ with degree greater than $1$.
		
		When $k=1$, the graph $G[U_1\setminus (\{u_1,u_3\}\cup V(M_1))]$ has only one connected component, which must be a star. Let the center of this star be $v_0$, and let two of its leaves be $v_1$ and $v_2$.
		
		If $\{w_1,\ldots,w_t\}\subseteq V(M_1)$, add the edge $v_0u_3$ to $M_2$. Since $v_0$ is the only vertex in $H$ with degree greater than $1$, $M_2$ covers all such vertices.
		
		If exactly one vertex in $\{w_1,\ldots,w_t\}$ is not in $V(M_1)$, assume without loss of generality that $w_1\notin V(M_1)$. By the choice of $M_1$, $w_1$ is adjacent to all vertices in $U_1\setminus (\{u_1,u_3\}\cup V(M_1))$, including $v_0$, $v_1$, and $v_2$. This implies that $w_1$ has at least $2n-11$ neighbors in $U_1$.
		
		If $w_1\in W_3$, by the definition of $W_3$, $w_1$ has at least $2n-11$ neighbors in $U_2$. Since $v_2$ has exactly one neighbor in $U_1$, it must have at least $2n-2-|W|\ge 2n-11$ neighbors in $U_2$. Thus, $w_1$ and $v_2$ share at least $2(2n-11)-|U_2|\ge 2n-22$ common neighbors in $U_2$. Let $u'$ be one such common neighbor. The sets $\{w_1,v_0,v_1\}$, $\{w_1,v_0,v_2\}$, $\{w_1,v_2,u'\}$, and $\{u_2,v_0,v_2\}$ each induce a triangle. Therefore, $G$ contains $H_a$ as a subgraph, a contradiction.
		
		If $w_1\in W_1$, then either $w_1u_2\in E(G)$ or $w_1u_4\in E(G)$. By symmetry, assume $w_1u_2\in E(G)$. The sets $\{w_1,v_0,v_1\}$, $\{w_1,v_0,v_2\}$, $\{w_1,v_2,u_2\}$, and $\{u_4,v_0,v_2\}$ each induce a triangle. Hence, $G$ contains $H_a$ as a subgraph, a contradiction.
	\end{proof}
	
	In the subgraph $G[(U_1\cup \{w_1,\ldots,w_t\})\setminus (\{u_1\}\cup V(M_1)\cup V(M_2))]$, each vertex has degree at most $1$, and this subgraph contains at least $2(n-|M_1|-|M_2|)$ vertices. Below, we demonstrate that $n-|M_1|-|M_2|\ge 2$.
	
	Since in the graph $G[U_1\setminus (\{u_1,u_3\}\cup V(M_1))]$, at most one-third of the vertices have degree greater than $1$, the vertex set $V(M_2)$ can contain at most two-thirds of this portion. Therefore, the remaining vertices number at least $|U_1\setminus (\{u_1,u_3\}\cup V(M_1))|/3\ge (2n-11)/3$. Combining this with $n\ge 111$, the subgraph $G[(U_1\cup \{w_1,\ldots,w_t\})\setminus (\{u_1\}\cup V(M_1)\cup V(M_2))]$ has at least $(222-11)/3\ge 70$ vertices.
	
	Hence, in the graph $\overline{G}[(U_1\cup \{w_1,\ldots,w_t\})\setminus (\{u_1\}\cup V(M_1)\cup V(M_2))]$, we can identify a matching consisting of $n-|M_1|-|M_2|$ edges, denoted by $M_3$. In $\overline{G}$, the edge set $M_1\cup M_2\cup M_3$ forms a matching of $n$ edges, and every vertex in this matching is adjacent to the vertex $u_1$. Consequently, $\overline{G}$ contains $F_n$ as a subgraph, leading to a final contradiction.
\end{proof}

\acknowledgements
The authors are grateful to the anonymous referee for helpful remarks. Y. Zhang was partially supported by the National Natural Science Foundation of China (NSFC) under Grant No. 11601527 and by the Natural Science Foundation of Hebei Province under Grant No. A2023205045.
\section*{Data Availability Statement}

No data was used or generated in this research.

\section*{Conflict of Interest}

The authors declare that they have no known competing financial interests or personal relationships that could have appeared to influence the work reported in this paper.

\nocite{*}
\bibliographystyle{abbrvnat}
\bibliography{reference}

\begin{thebibliography}{21}
\providecommand{\natexlab}[1]{#1}
\providecommand{\url}[1]{\texttt{#1}}
\expandafter\ifx\csname urlstyle\endcsname\relax
  \providecommand{\doi}[1]{doi: #1}\else
  \providecommand{\doi}{doi: \begingroup \urlstyle{rm}\Url}\fi

\bibitem[Bondy and Murty(2008)]{Bondy2008}
J.~A. Bondy and U.~S.~R. Murty.
\newblock \emph{Graph Theory}.
\newblock Springer, 2008.
\newblock \doi{10.1007/978-1-84628-970-5}.

\bibitem[Burr(1981)]{Burr1981}
S.~A. Burr.
\newblock Ramsey numbers involving graphs with long suspended paths.
\newblock \emph{Journal of the London Mathematical Society}, s2-24\penalty0 (3):\penalty0 405--413, 1981.
\newblock URL \url{https://doi.org/10.1112/jlms/s2-24.3.405}.

\bibitem[Burr and Erdős(1983)]{BurrErdos1983}
S.~A. Burr and P.~Erdős.
\newblock Generalizations of a {Ramsey}-theoretic result of chv\'atal.
\newblock \emph{Journal of Graph Theory}, 7\penalty0 (1):\penalty0 39--51, 1983.
\newblock URL \url{https://onlinelibrary.wiley.com/doi/abs/10.1002/jgt.3190070106}.

\bibitem[Burr et~al.(1983)Burr, Faudree, Rousseau, and Schelp]{Burr1983}
S.~A. Burr, R.~J. Faudree, C.~C. Rousseau, and R.~H. Schelp.
\newblock On {Ramsey} numbers involving starlike multipartite graphs.
\newblock \emph{Journal of Graph Theory}, 7\penalty0 (4):\penalty0 395--409, 1983.
\newblock URL \url{https://onlinelibrary.wiley.com/doi/abs/10.1002/jgt.3190070404}.

\bibitem[Chen et~al.(2021)Chen, Yu, and Zhao]{Chen2021}
G.~Chen, X.~Yu, and Y.~Zhao.
\newblock Improved bounds on the {Ramsey} number of fans.
\newblock \emph{European Journal of Combinatorics}, 96:\penalty0 103347, 2021.
\newblock \doi{10.1016/j.ejc.2021.103347}.
\newblock URL \url{https://www.sciencedirect.com/science/article/pii/S0195669821000391}.

\bibitem[Chung and Lin(2025)]{Chung2022}
F.~Chung and Q.~Lin.
\newblock Fan-complete {Ramsey} numbers.
\newblock \emph{Advances in Applied Mathematics}, 171:\penalty0 102939, 2025.
\newblock \doi{10.1016/j.aam.2025.102939}.
\newblock URL \url{https://www.sciencedirect.com/science/article/pii/S0196885825001010}.

\bibitem[Chv\'atal(1977)]{Chvatal1977}
V.~Chv\'atal.
\newblock Tree-complete graph {Ramsey} numbers.
\newblock \emph{Journal of Graph Theory}, 1\penalty0 (1):\penalty0 93, 1977.
\newblock URL \url{https://onlinelibrary.wiley.com/doi/abs/10.1002/jgt.3190010118}.

\bibitem[Dvořák and Metrebian(2023)]{Dvorak2023}
V.~Dvořák and H.~Metrebian.
\newblock A new upper bound for the {Ramsey} number of fans.
\newblock \emph{European Journal of Combinatorics}, 110:\penalty0 103680, 2023.
\newblock URL \url{https://doi.org/10.1016/j.ejc.2022.103680}.

\bibitem[Erdős et~al.(1995)Erdős, F\"uredi, Gould, and Gunderson]{Erdos1995}
P.~Erdős, Z.~F\"uredi, R.~J. Gould, and D.~S. Gunderson.
\newblock Extremal graphs for intersecting triangles.
\newblock \emph{Journal of Combinatorial Theory, Series B}, 64\penalty0 (1):\penalty0 89--100, 1995.
\newblock URL \url{https://www.sciencedirect.com/science/article/pii/S009589568571026X}.

\bibitem[Fan and Lin(2023)]{Fan2023}
C.~Fan and Q.~Lin.
\newblock Ramsey non-goodness involving books.
\newblock \emph{Journal of Combinatorial Theory, Series A}, 199:\penalty0 105780, 2023.
\newblock \doi{10.1016/j.jcta.2023.105780}.
\newblock URL \url{https://www.sciencedirect.com/science/article/abs/pii/S0097316523000481}.

\bibitem[Fan et~al.(2024)Fan, Huang, and Lin]{Fan2024}
C.~Fan, J.~Huang, and Q.~Lin.
\newblock Ramsey numbers of large books versus multipartite graphs.
\newblock \emph{Graphs and Combinatorics}, 40\penalty0 (6):\penalty0 125, 2024.
\newblock URL \url{https://link.springer.com/article/10.1007/s00373-024-02859-5}.

\bibitem[Fox et~al.(2023)Fox, He, and Wigderson]{Fox2023}
J.~Fox, X.~He, and Y.~Wigderson.
\newblock Ramsey goodness of books revisited.
\newblock \emph{Advances in Combinatorics}, 4:\penalty0 21, 2023.
\newblock \doi{10.19086/aic.2023.4}.
\newblock URL \url{https://www.advancesincombinatorics.com/article/83360-ramsey-goodness-of-books-revisited}.

\bibitem[Harary(1972)]{Harary1972}
F.~Harary.
\newblock Recent results on generalized {Ramsey} theory for graphs.
\newblock In Y.~Alavi et~al., editors, \emph{Graph Theory and Applications}, pages 125--138. Springer, Berlin, 1972.
\newblock URL \url{https://link.springer.com/chapter/10.1007/bfb0067364}.

\bibitem[Li and Rousseau(1996)]{Li1996}
Y.~Li and C.~C. Rousseau.
\newblock Fan-complete graph {Ramsey} numbers.
\newblock \emph{Journal of Graph Theory}, 23\penalty0 (4):\penalty0 413--420, 1996.
\newblock URL \url{https://onlinelibrary.wiley.com/doi/10.1002/(SICI)1097-0118(199612)23:4%3C413::AID-JGT10%3E3.0.CO;2-D}.

\bibitem[Lin and Li(2009)]{Lin2009}
Q.~Lin and Y.~Li.
\newblock On {Ramsey} numbers of fans.
\newblock \emph{Discrete Applied Mathematics}, 157\penalty0 (1):\penalty0 191--194, 2009.
\newblock \doi{10.1016/j.dam.2008.05.004}.
\newblock URL \url{https://www.sciencedirect.com/science/article/pii/S0166218X08002114}.

\bibitem[Lin and Liu(2021)]{Lin2021}
Q.~Lin and X.~Liu.
\newblock Ramsey numbers involving large books.
\newblock \emph{SIAM Journal on Discrete Mathematics}, 35\penalty0 (1):\penalty0 23--34, 2021.
\newblock URL \url{https://epubs.siam.org/doi/10.1137/20M1365867}.

\bibitem[Lin et~al.(2010)Lin, Li, and Dong]{Lin2010}
Q.~Lin, Y.~Li, and L.~Dong.
\newblock Ramsey goodness and generalized stars.
\newblock \emph{European Journal of Combinatorics}, 31\penalty0 (5):\penalty0 1228--1234, 2010.
\newblock \doi{10.1016/j.ejc.2009.10.011}.
\newblock URL \url{https://www.sciencedirect.com/science/article/pii/S0195669809002029}.

\bibitem[Liu and Li(2025)]{Liu2025}
M.~Liu and Y.~Li.
\newblock On graphs for which large books are {Ramsey} good.
\newblock \emph{Journal of Graph Theory}, 108\penalty0 (3), 2025.
\newblock URL \url{https://onlinelibrary.wiley.com/doi/abs/10.1002/jgt.23193}.

\bibitem[Nikiforov and Rousseau(2004)]{Nikiforov2004}
V.~Nikiforov and C.~C. Rousseau.
\newblock Large generalized books are $p$-good.
\newblock \emph{Journal of Combinatorial Theory, Series B}, 92\penalty0 (1):\penalty0 85--97, 2004.
\newblock URL \url{https://doi.org/10.1016/j.jctb.2004.03.009}.

\bibitem[Zeng et~al.(2025)Zeng, Zhang, and Zhao]{Zeng2025}
H.~Zeng, Y.~Zhang, and F.~Zhao.
\newblock Ramsey numbers of small wheels versus fans.
\newblock 2025.
\newblock submitted.

\bibitem[Zhang et~al.(2015)Zhang, Broersma, and Chen]{Zhang2015}
Y.~Zhang, H.~Broersma, and Y.~Chen.
\newblock A note on {Ramsey} numbers for fans.
\newblock \emph{Bulletin of the Australian Mathematical Society}, 92\penalty0 (1):\penalty0 19--23, 2015.
\newblock URL \url{https://doi.org/10.1017/S0004972715000398}.

\end{thebibliography}

\end{document}